\UseRawInputEncoding
\documentclass[12pt]{article}
\usepackage{amsfonts}
\usepackage{dsfont}
\usepackage{amsmath}
\usepackage{amstext}
\usepackage{amsopn}
\usepackage{amsbsy}
\usepackage{amscd}
\usepackage{amsxtra}
\usepackage{amsthm}
\usepackage{amssymb}

\numberwithin{equation}{section}

\title{Existence of solutions to the Poisson--Nernst--Planck system with singular permanent charges in $\mathbb{R}^2$}
\author{Chia-Yu Hsieh\footnote{Department of Mathematics and The Institute of Mathematical Sciences, The Chinese University of Hong Kong, Sha Tin, New Territories, Hong Kong; email: cyhsieh@math.cuhk.edu.hk} \hspace{2pt} and \hspace{2pt} Yong Yu \footnote{Department of Mathematics, The Chinese University of Hong Kong, Sha Tin, New Territories, Hong Kong; email: yongyu@math.cuhk.edu.hk; Y. Yu is partially supported by RGC grants of Hong Kong (14302718, 14306217).}}
\date{}

\newtheorem{thm}{Theorem}[section]
\newtheorem{defn}[thm]{Definition}
\newtheorem{lem}[thm]{Lemma}

\newtheorem{cor}[thm]{Corollary}
\newtheorem{rmk}[thm]{Remark}

\begin{document}
\maketitle

\begin{abstract}
In this paper, we study the well-posedness of Poisson--Nernst--Planck system with no-flux boundary condition and singular permanent charges in two dimension. The main difficulty comes from the lack of integrability of singular permanent charges. In order to overcome the difficulty, the main idea is to transform the system into another weighted parabolic system. By choosing suitable weighted spaces, local existence of solutions can be obtained based on a fixed-point argument. Moreover, we also proved global existence by energy estimates under the smallness assumption on initial data.
\end{abstract}

\noindent {\bf Mathematics Subject Classification.} 35K51, 35K57, 35K67.

\noindent {\bf Keywords.} Poisson--Nernst--Planck system, Singular permanent charge, Existence of solutions.

\section{Introduction}

The Poisson--Nernst--Planck (PNP) system is a basic model for the transport of charged particles. It has been applied in many applications both in physics and biology \cite{CSH17,CLE97,CI19,E98,H01,MRS90,MJP07,RCA89,RLW06,RLZ07,
WXLLS14,ZGLWS08}, sometimes modified by taking into account such as the sizes of ions, the interaction between ions and thermal effects \cite{GNE02,HHLLL15,HLLL20,KBA07,KBA07-2,LE14,LJX18,WLT16,XSL14}. The classical PNP system can be represented as
\begin{align}
\partial_t c_n &= \nabla \cdot \left( \nabla c_n - c_n \nabla \phi \right), \label{eqn_n}\\
\partial_t c_p &= \nabla \cdot \left( \nabla c_p + c_p \nabla \phi \right), \label{eqn_p}\\
\Delta \phi &= c_n - c_p - c_f \label{eqn_phi}
\end{align}
for $x \in \Omega$, $t > 0$, where $\Omega$ is a bounded domain with smooth boundary in $\mathbb{R}^2$. Here $c_n = c_n(x,t)$ and $c_p = c_p(x,t)$ represent the concentrations of anions and cations, respectively. $c_f = c_f(x)$ is the concentration of the permanent charge. $\phi = \phi(x,t)$ is the electric potential.

The system (\ref{eqn_n})--(\ref{eqn_phi}) consists of Nernst--Planck equations for ions $c_n$ and $c_p$, which state that the fluxes of ions are influenced by both the diffusion and the electric field, which is the gradient of the electric potential $\nabla \phi$. By Gauss' law, the electric potential $\phi$ can be determined by the charge density.  The well-posedness and long-time behavior of solutions to the PNP system have been studied by many mathematicians \cite{BD00,BHN94,G85,GG86,H19,HL15,M75}. To our best knowledge, all the known results concern the permanent charges with certain integrability or regularity, say, $c_f$ is in some $L^r$, $r > d$, where $d$ denotes the dimension of the spatial space, or H\"{o}lder continuous. Since point charges are usually described by Dirac measures, we are interested in the case that $c_f$ is an assemble of Dirac measures. Hence, in this work, we consider the PNP system (\ref{eqn_n})--(\ref{eqn_phi}) with
\begin{align}\label{singular_charge}
c_f = -2\pi \sum_{j=1}^m \alpha_j \delta_{x_j},
\end{align}
where $x_j \in \Omega$ are distinct points, $\alpha_j \neq 0$, $\delta_{x_j}$ denotes the Dirac measure centered at $x_j$, $j = 1, ..., m$.

\begin{rmk}
In \cite{BN97}, the authors considered a related problem with a single specie of ions and a permanent point charge fixed at the origin. In order to study the equation with singular charges, they adopted the notion of solutions obtained as limits of approximations, which is a different approach from ours.
\end{rmk}

We denote the initial data for $c_n$ and $c_p$ as
\begin{align}
c_n(x,0) &= c_{n,0}(x) \geq 0, \label{initial_n}\\
c_p(x,0) &= c_{p,0}(x) \geq 0, \label{initial_p}
\end{align}
for some non-negative functions $c_{n,0}$ and $c_{p,0}$. In order to describe an insulated boundary so that ions will not move across the boundary, it is natural to impose no-flux boundary conditions for $c_n$ and $c_p$. That is,
\begin{align}
\left( \nabla c_n - c_n \nabla \phi \right) \cdot \nu &= 0, \label{BC_n}\\
\left( \nabla c_p + c_p \nabla \phi \right) \cdot \nu &= 0 \label{BC_p}
\end{align}
for $x \in \partial\Omega$, $t > 0$, where $\nu$ is the unit outer normal of $\partial \Omega$. As for $\phi$, we use the Dirichlet boundary condition
\begin{align}
\phi = g \label{BC_phi}
\end{align}
for $x \in \partial\Omega$, $t > 0$, for some given smooth function $g$, to describe the given surface potential. Notice that the initial for $\phi$, i.e., $\phi(x,0)$ can be determined by (\ref{eqn_phi}) and (\ref{BC_phi}).

\begin{rmk}
The boundary condition for the electric potential $\phi$ depends on what information we have on the boundary. Usually, it is given by a Dirichlet boundary condition describing a given surface potential, or a Neumann boundary condition describing a given surface charge distribution. While in some literatures, Robin type boundary conditions are used to take into account of the capacitance effects (see, for example, \cite{Lee16}).
\end{rmk}

When there is no permanent charge $c_f = 0$, the PNP system (\ref{eqn_n})--(\ref{eqn_phi}) supplemented with the initial and boundary conditions (\ref{initial_n})--(\ref{BC_phi}) is studied in \cite{BHN94}. The authors established global existence for two dimensional spatial space and local existence for higher dimensional case. The purpose of this work is to study existence of solutions to PNP system with singular permanent charges (\ref{eqn_n})--(\ref{BC_phi}). The main difficulty comes from point singularities in the Poisson equation (\ref{eqn_phi}).

It is natural to deal with Dirac measures by using the Green function. That is, by letting $G$ be a function satisfying $\Delta G = 2\pi \sum_{j=1}^m \alpha_j \delta_{x_j}$, then we consider $\psi := \phi - G$, which solves the Poisson equation $\Delta \psi = c_n - c_p$ without Dirac measures. However, under the transformation, there exists the coefficient $\nabla G$ in the drift terms of the equations of $c_n$ and $c_p$. Notice that $\nabla G \in L^r$ only for $r \in [1,2)$. Because of the lack of integrability of the coefficient $\nabla G$, the problem cannot be solved by the parabolic theory in the usual Sobolev spaces. In order to overcome the difficulty, we reformulate the problem to be a parabolic system with weight functions, which are possibly singular and degenerate at some points. Furthermore, by choosing weighted Sobolev spaces with suitable embedding and compactness, we can establish the well-posedness by using Schauder's fixed-point theorem.

\begin{rmk}\label{rmk_1}
Stationary solutions of (\ref{eqn_n})--(\ref{eqn_phi}) with (\ref{BC_n})--(\ref{BC_phi}) satisfy the charge-conserving Poisson--Boltzmann (CCPB) equation (see \cite{HY20,Lee16,LLHLL11,LR18}). In \cite{HY20}, the authors proved that the CCPB equation with the permanent charges $2\pi \sum_{j=1}^m \alpha_j \delta_{x_j}$ admits a unique solution for given boundary condition and permanent charges provided $|\alpha_j| \leq 2$ for $j = 1, ..., m$, in two-dimensional space. By contrast, for higher-dimensional case $\Omega \subset \mathbb{R}^d$, $d \geq 3$, all isolated singularities of solutions of CCPB equation are removable (see Proposition 1.1 in \cite{HY20}). It motivates us to study this problem in $\mathbb{R}^2$.
\end{rmk}

Now we introduce the reformulation as follows.

\subsection{Reformulation}\label{sec_1_1}

First, we use Green function to replace the Dirac measures. Let
\begin{align}\label{def_G}
G(x) = \sum_{j = 1}^m \alpha_j \log{|x - x_j|} + H(x),
\end{align}
where $H$ satisfies
\begin{align}\label{def_H}
\begin{cases}
\Delta H = 0 \quad &\mbox{in } \Omega, \\[3mm]
H = -\displaystyle\sum_{j = 1}^m \alpha_j \log{|x - x_j|} + g \quad &\mbox{on } \partial\Omega.
\end{cases}
\end{align}
That is, $G$ solves
\begin{align*}
\begin{cases}
\Delta G = 2\pi \displaystyle\sum_{j=1}^m \alpha_j \delta_{x_j} \quad &\mbox{in } \Omega, \\[3mm]
G = g \quad &\mbox{on } \partial\Omega.
\end{cases}
\end{align*}
Therefore, we can decompose $\phi$ into two parts $\phi = \psi + G$ such that the equation for $\psi$ does not involve Dirac measures. In the meantime, as we introduced before, we further transform $(c_n,c_p)$ to $(u,v)$ by multiplying weighted functions to $(c_n,c_p)$ as
\begin{align}
u(x,t) &= e^{-G(x)} c_n(x,t) = \prod_{j = 1}^m |x - x_j|^{-\alpha_j} e^{-H(x)} c_n(x,t), \label{def_u}\\
v(x,t) &= e^{G(x)} c_p(x,t) = \prod_{j = 1}^m |x - x_j|^{\alpha_j} e^{H(x)} c_p(x,t), \label{def_v}
\end{align}
and we have
\begin{align}
\psi(x,t) &= \phi(x,t) - G(x). \label{def_psi}
\end{align}
Using the equations of $(c_n,c_p,\phi)$, we obtain the equations of $(u,v,\psi)$ as follows:
\begin{align}
\prod_{j = 1}^m |x - x_j|^{\alpha_j} e^{H(x)} u_t &= \nabla \cdot \left[ \prod_{j = 1}^m |x - x_j|^{\alpha_j} e^{H(x)} (\nabla u - u \nabla \psi) \right] \label{eqn_u}\\
\prod_{j = 1}^m |x - x_j|^{-\alpha_j} e^{-H(x)} v_t &= \nabla \cdot \left[ \prod_{j = 1}^m |x - x_j|^{-\alpha_j} e^{-H(x)} (\nabla v + v \nabla \psi) \right] \label{eqn_v}\\
\Delta \psi &= \prod_{j = 1}^m |x - x_j|^{\alpha_j} e^{H(x)} u - \prod_{j = 1}^m |x - x_j|^{-\alpha_j} e^{-H(x)} v \label{eqn_psi}
\end{align}
for $x \in \Omega$, $t > 0$.
The corresponding initial and boundary conditions become
\begin{align}
\left( \nabla u - u \nabla \psi \right) \cdot \nu &= 0 \label{BC_u}\\
\left( \nabla v + v \nabla \psi \right) \cdot \nu &= 0 \label{BC_v}\\
\psi &= 0 \label{BC_psi}
\end{align}
for $x \in \partial\Omega$, $t > 0$, and
\begin{align}
u(x,0) &= u_0(x) = \prod_{j = 1}^m |x - x_j|^{-\alpha_j} e^{-H(x)} c_{n,0}(x), \label{initial_u}\\
v(x,0) &= v_0(x) = \prod_{j = 1}^m |x - x_j|^{\alpha_j} e^{H(x)} c_{p,0}(x). \label{initial_v}
\end{align}
Also,
\begin{align}\label{initial_psi}
\psi(x,0) = \psi_0(x)
\end{align}
is determined by (\ref{eqn_psi}) and (\ref{BC_psi}). For simplicity, we let
\begin{align}\label{def_w}
w(x) = \prod_{j = 1}^m |x - x_j|^{\alpha_j} e^{H(x)},
\end{align}
and then equations (\ref{eqn_u})--(\ref{eqn_psi}) can be written as
\begin{align}
w u_t &= \nabla \cdot \left[ w (\nabla u - u \nabla \psi) \right], \label{new_eqn_u}\\
w^{-1} v_t &= \nabla \cdot \left[ w^{-1} (\nabla v + v \nabla \psi) \right], \label{new_eqn_v}\\
\Delta \psi &= w u - w^{-1} v. \label{new_eqn_psi}
\end{align}
Notice that $w$ and $w^{-1}$ have singularities and zeros. It is natural to look for solutions in weighted spaces. We introduce some notations for later use. For locally summable non-negative function $\mu$ on a domain $E$, $k \in \mathbb{N} \cup \{0\}$, $1 \leq p < \infty$, let $W_\mu^{k,p}(E)$ denote the weighted space $W^{k,p}(E,\mu dx)$. Also, we write $L_\mu^p(E) = W_\mu^{0,p}(E)$ and $H_\mu^k(E) = W_\mu^{k,2}(E)$. And the dual space of $H_\mu^1(E)$ is denoted by $H_\mu^{-1}(E)$.

Now, we define the weak solutions of the system (\ref{eqn_u})--(\ref{initial_psi}) in weighted spaces as follows.

\begin{defn}\label{def_sol}
For $T > 0$, $(u,v,\psi)$ is a solution of (\ref{eqn_u})--(\ref{initial_psi}) if
\begin{align*}
u \in L^2\left((0,T);H_w^1(\Omega)\right), \quad v \in L^2\left((0,T);H_{1/w}^1(\Omega)\right),
\end{align*}
and
\begin{align*}
u_t \in L^2\left((0,T);H_w^{-1}(\Omega)\right), \quad v_t \in L^2\left((0,T);H_{1/w}^{-1}(\Omega)\right),
\end{align*}
such that
\begin{align*}
&\quad\ \int_\Omega u(x,t) \xi(x,t) w(x) dx - \int_0^t \int_\Omega u \xi_t w dxdt + \int_0^t \int_\Omega (\nabla u - u \nabla \psi) \cdot \nabla \xi w dxdt \\
&= \int_\Omega u_0(x) \xi(x,0) w(x) dx, \\
&\quad\ \int_\Omega v(x,t) \eta(x,t) w^{-1}(x) dx - \int_0^t \int_\Omega v \eta_t w^{-1} dxdt + \int_0^t \int_\Omega (\nabla v + v \nabla \psi) \cdot \nabla \eta w^{-1} dxdt \\
&= \int_\Omega v_0(x) \eta(x,0) w^{-1}(x) dx, \\
\end{align*}
for each test function $\xi \in H^1\left((0,T);H_w^1(\Omega)\right)$, $\eta \in H^1\left((0,T);H_{1/w}^1(\Omega)\right)$ and almost every $t \in [0,T]$, where $\psi$ is the weak solution of (\ref{eqn_psi}) and (\ref{BC_psi}).
\end{defn}

This is a modification of the standard definition of weak solutions of parabolic equations (see Chapter III, Section 5 in \cite{LSU68}) for no-flux boundary conditions in weighted spaces.

In the following, we state our main results.

\subsection{Main results}

Our first result concerns with the local well-posedness of the problem (\ref{eqn_u})--(\ref{initial_psi}).

\begin{thm}\label{local_wellposedness}
Given non-negative initial data $u_0 \in L_w^2(\Omega)$ and $v_0 \in L_{1/w}^2(\Omega)$ and assuming that $0 < |\alpha_j| < 2\sqrt{2} - 2$, $j = 1, ..., m$, there exists $T = T\left(\|u_0\|_{L_w^2(\Omega)},\|v_0\|_{L_{1/w}^2(\Omega)}\right) > 0$ such that the system (\ref{eqn_u})--(\ref{initial_psi}) has a unique weak solution $(u,v,\psi)$ in $(0,T)$ with
\begin{align*}
u \in L^2\left((0,T);H_w^1(\Omega)
\right)\, \cap\, L^\infty((0,T);L_w^2(\Omega))
\end{align*}
and
\begin{align*}
v \in L^2\left((0,T);H_{1/w}^1(\Omega)\right)\, \cap\, L^\infty\left((0,T);L_{1/w}^2(\Omega)\right).
\end{align*}
Moreover, $u$ and $v$ are non-negative.
\end{thm}

In order to obtain the global existence, it suffices to show that the weighted $L^2$-norms of $u$ and $v$ are bounded in any finite time. A standard argument by using the equations of $u$ and $v$ is to control the integrals
\begin{align}\label{integrals}
\int_\Omega w u \nabla \psi \cdot \nabla u dx \quad \mbox{and} \quad \int_\Omega w^{-1} v \nabla \psi \cdot \nabla v dx
\end{align}
similar to (\ref{id_01})--(\ref{id_02}) for the linearized problem in Section 2. Without permanent charges, the authors in \cite{BHN94} proved the global existence in two dimension by employing the energy law for system (\ref{eqn_n})--(\ref{eqn_phi})
\begin{align*}
&\quad\ \frac{d}{dt} \int_\Omega \left( c_n \log c_n + c_p \log c_p + \frac{1}{2} | \nabla \phi |^2 \right) dx \\
&= -\int_\Omega \left( c_n | \nabla ( \log c_n - \phi ) |^2 + c_p | \nabla ( \log c_p + \phi ) |^2 \right) dx \leq 0.
\end{align*}
The uniform boundedness of $\| c_n \log c_n \|_{L^1(\Omega)}$, $\| c_p \log c_p \|_{L^1(\Omega)}$ and $\| \nabla \phi \|_{L^2(\Omega)}$ is crucial to control the integrals (\ref{integrals}) with $w \equiv 1$. On the other hand, for the case with singular permanent charges, the situation is different. First, there is a similar energy law, but with weight functions in the integrals. Therefore the weighted Sobolev spaces should be employed, while the embedding theorems (cf. Corollary \ref{lem_2_2}) are much weaker than those for standard Sobolev spaces. Second, since there are weights $w$ and $w^{-1}$ on the right-hand side of the Poisson equation (\ref{new_eqn_psi}), more integrability is needed when norms with different weights are alternatively used in elliptic estimates.

Therefore, we assume that weighted norms of the initial data are sufficiently small to allow us to deduce a global-in-time energy estimate. We have the following result for global existence.

\begin{thm}\label{global_wellposedness}
With the same assumptions as in Theorem \ref{local_wellposedness}, there is a positive constant $\epsilon > 0$ such that if, in addition,
\begin{align}\label{smallness_assumption}
\| u_0 \|_{L_w^2(\Omega)}^2 + \| v_0 \|_{L_{1/w}^2(\Omega)}^2 < \epsilon,
\end{align}
then the solution to the system (\ref{eqn_u})--(\ref{initial_psi}) exists globally in time.
\end{thm}

\begin{rmk}
In the original formulation, the smallness assumption (\ref{smallness_assumption}) is equivalent to
\begin{align*}
\| c_{n,0} \|_{L_{1/w}^2(\Omega)}^2 + \| c_{p,0} \|_{L_w^2(\Omega)}^2 < \epsilon.
\end{align*}
Notice that $w$ approaches infinity at locations of positive point charges and goes to zero at locations of negative point charges. Thus, anions can be more concentrated near positive point charges, and should be less concentrated near negative point charges. Similarly, an opposite conclusion holds for cations.
\end{rmk}

\subsection{Remarks and discussion}

\begin{itemize}
\item[(i)] If we apply a similar reformulation as in Section \ref{sec_1_1} for higher dimensional cases $\Omega \subset \mathbb{R}^d$, $d \geq 3$, the resulting system is similar to (\ref{eqn_u})--(\ref{initial_psi}), but with much more singular weight functions with exponential singularities. Therefore, the method we used in this article cannot work for higher dimensional cases.
\item[(ii)] In this article, all physical coefficients are scaled to be one in the PNP system for simplicity. The original PNP model with permanent charge (\ref{singular_charge}) can be represented as
\begin{align}
\partial_t c_n &= D_n \nabla \cdot \left( \nabla c_n - \frac{z_n e}{k_B \mathcal{T}} c_n \nabla \phi \right), \label{original_eqn_1}\\
\partial_t c_p &= D_p \nabla \cdot \left( \nabla c_p + \frac{z_p e}{k_B \mathcal{T}} c_p \nabla \phi \right), \label{original_eqn_2}\\
\varepsilon \Delta \phi &= z_n e c_n - z_p e c_p + 2\pi \sum_{j=1}^m \alpha_j \delta_{x_j}, \label{original_eqn_3}
\end{align}
where $D_n$, $D_p$, $z_n$, $z_p$, $k_B$, $\mathcal{T}$ and $\varepsilon$ are the diffusion coefficient of anions, the diffusion coefficient of cations, the valence of anions, the valence of cations, the Boltzmann constant, the temperature and the dielectric constant, respectively. One can see that, if we consider the original PNP model (\ref{original_eqn_1})--(\ref{original_eqn_3}) with all these physical coefficients, the upper bound $2\sqrt{2} - 2$ for $a_j$'s in Theorem \ref{local_wellposedness} will become $(2\sqrt{2} - 2)\varepsilon$.
\item[(iii)] Because of the presence of singular permanent charges, it is not straightforward to define weak solutions to the problem (\ref{eqn_n})--(\ref{BC_phi}). Instead of adopting the notion of solutions obtained as limits of approximations as in \cite{BN97}, we define the solutions by considering a reformulated problem in suitable weighted spaces. In view of Remark \ref{rmk_1}, one might expect that the assumption in Theorem \ref{local_wellposedness} on $\alpha_j$'s can be weakened to be $|\alpha_j| \leq 2$, $j = 1, ..., m$. However, there are some limitations of weighted spaces. First, the weighted embeddings are worse than the standard ones. Also, in order to estimate weighted norms of $\psi$, we lose some integrability because $u$ and $v$ are multiplied by different weight functions on the right-hand side of the Poisson equation (\ref{new_eqn_psi}). Due to these limitations, we can only prove the existence result under the assumption $|\alpha_j| < 2\sqrt{2} - 2$, $j = 1, ..., m$.
\end{itemize}

The rest of this paper is organized as follows. In Section \ref{sec_2_1}, we will introduce the weighted spaces as well as their properties employed to study the system (\ref{eqn_u})--(\ref{initial_psi}). Then we prove Theorem \ref{local_wellposedness} in Section \ref{sec_2_2}. Finally, Theorem \ref{global_wellposedness} will be proved in Section \ref{sec_3}. We will use $C$ to denote generic constants that are independent of particular functions appearing in inequalities and write $A \lesssim B$ if $A \leq CB$.

\section{Local well-posedness}

In this section, we will prove Theorem \ref{local_wellposedness}. The local existence can be established by Schauder's fixed-point theorem. Before we prove the existence result, we introduce the weighted spaces and some properties we will use.

\subsection{Weighted Sobolev spaces}\label{sec_2_1}

First, we introduce the Muckenhoupt $A_p$ weights. In $\mathbb{R}^d$, $d \in \mathbb{N}$, for $1 < p < \infty$, the class $A_p = A_p(\mathbb{R}^d)$ of weight functions consists of those weights $\mu > 0$, $\mu \in L^1_{\mathrm{loc}}(\mathbb{R}^d)$ such that
\begin{align*}
\left( \frac{1}{|B|} \int_B \mu dx \right) \left( \frac{1}{|B|} \int_B \mu^{-\frac{1}{p-1}} dx \right)^{p-1} \leq C
\end{align*}
for all balls $B \subset \mathbb{R}^d$ for some positive constant $C$. It is first introduced by Muckenhoupt \cite{M72,M74} (see also monographs \cite{HKM06} and \cite{S93}) that the class consists of those weights $\mu$ such that the maximal operator is bounded in $L^p_\mu(\mathbb{R}^d)$. One can verify by using the definition that
\begin{align*}
A_q \subset A_p
\end{align*}
for all $1 < q \leq p$.

A basic example of $A_p$ weights is the power weight $|x|^\alpha$. We have $|x|^\alpha \in A_p$ if and only if $-d < \alpha < d(p-1)$. For the weights $w$ and $1/w$ defined by (\ref{def_w}), we may extend the domain of $w$ to the whole space $\mathbb{R}^2$ by assuming $w = 1$ on $\Omega^c$. Notice that, under the assumptions in Theorem \ref{local_wellposedness}, singularities and zeros of $w$ and $1/w$ have orders $|x|^\alpha$ with $|\alpha| < 2\sqrt{2} - 2$. It is easy to check that $w$ and $1/w$ belong to $A_p$ class for all $p > 2\sqrt{2} - 2$.

In the following, we consider weighted Sobolev spaces with $A_p$ weights. Theory of weighted Sobolev spaces with $A_p$ weights has been studied by mathematicians. In the pioneer work \cite{FKS82}, the authors proved the Poincar\'{e} inequality for weighted Sobolev spaces in balls as follows. Given a ball $B \subset \mathbb{R}^d$, there exist positive constants $C$ and $\delta$ such that
\begin{align}\label{poincare_1}
\| h - h_B \|_{L^{kp}_\mu(B)} \leq C \| \nabla h \|_{L^p_\mu(B)}
\end{align}
for all $h \in W^{1,p}_\mu(B)$ and $1 \leq k \leq \dfrac{n}{n-1} + \delta$, where
\begin{align*}
h_B = \frac{1}{\mu(B)} \int_B u(x) \mu(x) dx.
\end{align*}
The relation between the constant $k$ and the weight function $\mu$ in (\ref{poincare_1}) can be given more precisely. That is,

\begin{lem}\label{lem_2_1}
Let $E \subset \mathbb{R}^d$ be a bounded domain with smooth boundary. Suppose that $\mu \in A_p$ and let $1 \leq p_0 < q < p < dq$ and $k = d / ( d - p/q ) > 1$, where
\begin{align*}
p_0 = \inf_{p > 1} \left\{ \mu \in A_p \right\}.
\end{align*}
Then there exists a positive constant $C$ such that
\begin{align}\label{poincare_ineq}
\| h - h_E \|_{L^{kp}_\mu(E)} \leq C \| \nabla h \|_{L^p_\mu(E)}
\end{align}
for all $h \in W^{1,p}_\mu(E)$, where
\begin{align*}
h_E = \frac{1}{\mu(E)} \int_E h(x) \mu(x) dx.
\end{align*}
\end{lem}

\begin{proof}
By Theorem 15.26 of \cite{HKM06}, the inequality (\ref{poincare_ineq}) holds if $E$ is a ball. Now, we choose a sufficiently large ball $B$ such that $E \subset B$. By the extension theorem for weighted Sobolev spaces (see Theorem 1.4 in \cite{C92}), there exists an extension operator $\Lambda : W_\mu^{1,p}(E) \rightarrow W_\mu^{1,p}(B)$, such that
\begin{align*}
\| \Lambda h \|_{L_\mu^p(B)} \leq C \| h \|_{L_\mu^p(E)} \quad \mbox{and} \quad \| \nabla (\Lambda h) \|_{L_\mu^p(B)} \leq C \| \nabla h \|_{L_\mu^p(E)} 
\end{align*}
for any $h \in W_\mu^{1,p}(E)$. Here $\mu$ is extend to be $1$ on $B \backslash E$. Combining Lemma 1 in \cite{R19}, the extension theorem and inequality (\ref{poincare_ineq}) for balls, we have
\begin{align*}
\| h - h_E \|_{L^{kp}_\mu(E)} &\leq 2 \inf_{c \in \mathbb{R}} \| h - c \|_{L^{kp}_\mu(E)} \leq 2 \| \Lambda h - (\Lambda h)_B \|_{L^{kp}_\mu(B)} \\
&\leq C \| \nabla (\Lambda h) \|_{L^p_\mu(B)} \leq C \| \nabla h \|_{L_\mu^p(E)}.
\end{align*}
\end{proof}

A direct consequence of Lemma \ref{lem_2_1} is the Sobolev inequality.

\begin{cor}\label{lem_2_2}
Under the same assumptions of Lemma \ref{lem_2_1}, there exists a positive constant $C$ such that
\begin{align*}
\| h \|_{L^{kp}_\mu(E)} \leq C \| h \|_{W^{1,p}_\mu(E)}
\end{align*}
for all $h \in W^{1,p}_\mu(E)$.
\end{cor}

Recall that $w , 1/w \in A_q$ for all $q > 2\sqrt{2} - 2$. By applying Corollary \ref{lem_2_2} with $\mu = w , 1/w$ and $p = 2$, we can choose $q$ greater than and sufficiently close to $2\sqrt{2} - 2$ to conclude that there is $\sigma > 0$ such that
\begin{align}\label{sobolev_ineq}
\| h \|_{L^{2\sqrt{2} / (\sqrt{2} - 1) + \sigma}_\mu(\Omega)} \leq C \| h \|_{H^1_\mu(\Omega)}
\end{align}
for all $h \in H^1_\mu(\Omega)$, where $\mu = w$ or $1/w$. By using an interpolation inequality and (\ref{sobolev_ineq}), we have
\begin{align}
\| h \|_{L^{2\sqrt{2} / (\sqrt{2} - 1)}_\mu(\Omega)} &\leq \| h \|_{L^{2\sqrt{2} / (\sqrt{2} - 1) + \sigma}_\mu(\Omega)}^{1-\theta} \| h \|_{L^2_\mu(\Omega)}^\theta \notag\\
&\leq C \| h \|_{H^1_\mu(\Omega)}^{1-\theta} \| h \|_{L^2_\mu(\Omega)}^\theta \label{interpolation}
\end{align}
for all $h \in H^1_\mu(\Omega)$, $\mu = w$ or $1/w$. Here $\theta \in (0,1)$ satisfies
\begin{align}\label{def_theta}
\dfrac{\sqrt{2} - 1}{2\sqrt{2}} = \dfrac{1 - \theta}{\dfrac{2\sqrt{2}}{\sqrt{2} - 1} + \sigma} + \dfrac{\theta}{2}.
\end{align}

Moreover, compact embeddings are needed in order to apply the Aubin-Lions lemma. The compact embedding theorem for weighted Sobolev spaces with $A_p$ weights is studied in \cite{F00}. The author proved the following result.

\begin{lem}\label{lem_2_3}
Let $E \subset \mathbb{R}^d$ be a bounded smooth domain. If $\mu \in A_p$, $1 < p < \infty$, then we have the compact embedding
\begin{align*}
W^{1,p}_\mu(E) \subset \subset L^p_\mu(E).
\end{align*}
\end{lem}

Finally, noticing that there are weight functions on the right-hand side of the Poisson equation (\ref{new_eqn_psi}), we will employ the weighted elliptic estimate given in \cite{DST08} to estimate $\psi$ in terms of $u$ and $v$. That is,

\begin{lem}\label{lem_2_4}
Let $E \subset \mathbb{R}^d$ be a bounded smooth domain. If $h$ is the solution of the Poisson equation
\begin{align*}
\begin{cases}
\Delta h = f \quad & \mbox{in } E, \\[3mm]
h = 0 \quad & \mbox{on } \partial E,
\end{cases}
\end{align*}
where $f \in L^p_w(E)$ and $w \in A_p$ for some $1 < p < \infty$, then there exists a positive constant $C$ depending only on $E$ and $w$ such that
\begin{align*}
\| h \|_{W^{2,p}_w(E)} \leq C \| f \|_{L^p_w(E)}.
\end{align*}
\end{lem}

\subsection{Proof of Theorem \ref{local_wellposedness}}\label{sec_2_2}

In this section, we prove the existence and uniqueness of local solutions. We will apply Schauder's fixed-point theorem to find a local solution, which is a fixed point of the mapping defined by $F(\bar{u},\bar{v}) = (u,v)$ in $X = L^{2/\theta}\left((0,T);L_w^2(\Omega)\right) \times L^{2/\theta}\left((0,T);L_{1/w}^2(\Omega)\right)$ for some $T > 0$ to be chosen later and $\theta$ given in (\ref{def_theta}). Here in the definition of $F$, given $(\bar{u},\bar{v})$, $(u,v)$ is the solution of
\begin{align}
w u_t &= \nabla \cdot \left[ w (\nabla u - u \nabla \bar{\psi}) \right], \label{eqn_u_bar}\\
w^{-1} v_t &= \nabla \cdot \left[ w^{-1} (\nabla v + v \nabla \bar{\psi}) \right], \label{eqn_v_bar}
\end{align}
with boundary conditions
\begin{align}
(\nabla u - u \nabla \bar{\psi}) \cdot \nu &= 0, \label{BC_u_bar}\\
(\nabla v + v \nabla \bar{\psi}) \cdot \nu &= 0, \label{BC_v_bar}
\end{align}
and initial data $(u_0,v_0)$ as in (\ref{initial_u})--(\ref{initial_v}),
where $\bar{\psi}$ is determined by
\begin{align}
\Delta \bar{\psi} &= w \bar{u} - w^{-1} \bar{v} \qquad\mbox{in } \Omega, \label{eqn_psi_bar}\\
\bar{\psi} &= 0 \qquad \mbox{on } \partial\Omega, \label{BC_psi_bar}
\end{align}
for almost every $t \in [0,T]$.

First, we claim that $F$ is well-defined. The existence and uniqueness of solutions of the system (\ref{eqn_u_bar})--(\ref{BC_v_bar}) with initial data (\ref{initial_u})--(\ref{initial_v}) can be obtain by using analogous arguments in the linear theory of the parabolic equations \cite{LSU68} and employing weighted estimates in the weighted Sobolev spaces of weights $w$ and $w^{-1}$ (see the sketch in the Appendix). Moreover, the solution $(u,v)$ satisfies
\begin{align}
\frac{1}{2} \frac{d}{dt} \int_\Omega wu^2 dx &= -\int_\Omega w|\nabla u|^2 dx + \int_\Omega w u \nabla \bar{\psi} \cdot \nabla u dx, \label{id_01}\\
\frac{1}{2} \frac{d}{dt} \int_\Omega w^{-1} v^2 dx &= -\int_\Omega w^{-1} |\nabla v|^2 dx - \int_\Omega w^{-1} v \nabla \bar{\psi} \cdot \nabla v dx. \label{id_02}
\end{align}
By H\"{o}lder's inequality and the interpolation inequality (\ref{interpolation}), the last integrals in (\ref{id_01}) and (\ref{id_02}) can be estimated by
\begin{align}
\left| \int_\Omega w u \nabla \bar{\psi} \cdot \nabla u dx \right| &\leq \| \nabla u \|_{L_w^2(\Omega)} \| u \|_{L_w^{2\sqrt{2} / (\sqrt{2} - 1)}(\Omega)} \| \nabla \bar{\psi} \|_{L_w^{2\sqrt{2}}(\Omega)}, \notag\\
&\lesssim \| \nabla u \|_{L_w^2(\Omega)} \| u \|_{H_w^1(\Omega)}^{1-\theta} \| u \|_{L_w^2(\Omega)}^\theta \| \nabla \bar{\psi} \|_{L_w^{2\sqrt{2}}(\Omega)}, \label{id_03}\\
\left| \int_\Omega w^{-1} v \nabla \bar{\psi} \cdot \nabla v dx \right| &\leq \| \nabla v \|_{L_{1/w}^2(\Omega)} \| v \|_{L_{1/w}^{2\sqrt{2} / (\sqrt{2} - 1)}(\Omega)} \| \nabla \bar{\psi} \|_{L_{1/w}^{2\sqrt{2}}(\Omega)} \notag\\
&\lesssim \| \nabla v \|_{L_{1/w}^2(\Omega)} \| v \|_{H_{1/w}^1(\Omega)}^{1-\theta} \| v \|_{L_{1/w}^2(\Omega)}^\theta \| \nabla \bar{\psi} \|_{L_{1/w}^{2\sqrt{2}}(\Omega)}. \label{id_04}
\end{align}
We will make use of Lemma \ref{lem_2_4} to estimate the weighted norms of $\bar{\psi}$ in the last two inequalities. Since $\bar{\psi}$ satisfies (\ref{eqn_psi_bar})--(\ref{BC_psi_bar}), it can be decomposed into two parts $\bar{\psi} = \bar{\psi}_1 - \bar{\psi}_2$, where $\bar{\psi}_1$ and $\bar{\psi}_2$ solve
\begin{align*}
\begin{cases}
\Delta \bar{\psi}_1 = w \bar{u} \quad &\mbox{in } \Omega,\\[3mm]
\bar{\psi}_1 = 0 \quad &\mbox{on } \partial\Omega,
\end{cases}
\quad \mbox{and} \quad
\begin{cases}
\Delta \bar{\psi}_2 = w^{-1} \bar{v} \quad &\mbox{in } \Omega,\\[3mm]
\bar{\psi}_2 = 0 \quad &\mbox{on } \partial\Omega,
\end{cases}
\end{align*}
respectively. For almost every $t \in [0,T]$, $\bar{u} \in L_w^2(\Omega)$ and $\bar{v} \in L_{1/w}^2(\Omega)$. Moreover, we have
\begin{align*}
\|w\bar{u}\|_{L_{1/w}^2(\Omega)}  = \|\bar{u}\|_{L_w^2(\Omega)} \quad\mbox{and}\quad \|w^{-1}\bar{v}\|_{L_w^2(\Omega)} = \|\bar{v}\|_{L_{1/w}^2(\Omega)}.
\end{align*}
By H\"{o}lder's inequality, Sobolev inequality (\ref{sobolev_ineq}) and Lemma \ref{lem_2_4},
\begin{align*}
\| \nabla \bar{\psi}_1 \|_{L_w^{2\sqrt{2}}(\Omega)} &\leq \| \nabla \bar{\psi}_1 \|_{L_{1/w}^{2\sqrt{2} / (\sqrt{2} - 1)}(\Omega)} \| w \|_{L^{1/(\sqrt{2} - 1)}(\Omega)}^{1/2} \\
&\lesssim \| \bar{\psi}_1 \|_{H^2_{1/w}(\Omega)} \\
&\lesssim \| \bar{u} \|_{L^2_w(\Omega)}.
\end{align*}
As for $\bar{\psi}_2$, we have
\begin{align*}
\| \nabla \bar{\psi}_2 \|_{L_w^{2\sqrt{2}}(\Omega)} &\lesssim \| \bar{\psi}_2 \|_{H^2_w(\Omega)} \\
&\lesssim \| \bar{v} \|_{L^2_{1/w}(\Omega)}.
\end{align*}
Combining the last two inequalities, we obtain
\begin{align}\label{id_16}
\| \nabla \bar{\psi} \|_{L_w^{2\sqrt{2}}(\Omega)} \lesssim \| \bar{u} \|_{L^2_w(\Omega)} + \| \bar{v} \|_{L^2_{1/w}(\Omega)}.
\end{align}
Similarly, we have
\begin{align}\label{id_17}
\| \nabla \bar{\psi} \|_{L_{1/w}^{2\sqrt{2}}(\Omega)} \lesssim \| \bar{u} \|_{L^2_w(\Omega)} + \| \bar{v} \|_{L^2_{1/w}(\Omega)}.
\end{align}
Putting (\ref{id_16})--(\ref{id_17}) into (\ref{id_03})--(\ref{id_04}), we have
\begin{align}
\left| \int_\Omega w u \nabla \bar{\psi} \cdot \nabla u dx \right| &\lesssim \| \nabla u \|_{L_w^2(\Omega)} \| u \|_{H_w^1(\Omega)}^{1-\theta} \| u \|_{L_w^2(\Omega)}^\theta \left( \| \bar{u} \|_{L^2_w(\Omega)} + \| \bar{v} \|_{L^2_{1/w}(\Omega)} \right), \label{id_18}\\
\left| \int_\Omega w^{-1} v \nabla \bar{\psi} \cdot \nabla v dx \right| &\lesssim \| \nabla v \|_{L_{1/w}^2(\Omega)} \| v \|_{H_{1/w}^1(\Omega)}^{1-\theta} \| v \|_{L_{1/w}^2(\Omega)}^\theta \left( \| \bar{u} \|_{L^2_w(\Omega)} + \| \bar{v} \|_{L^2_{1/w}(\Omega)} \right). \label{id_19}
\end{align}
By Young's inequality, (\ref{id_18})--(\ref{id_19}) imply that
\begin{align}
&\quad\ \left| \int_\Omega w u \nabla \bar{\psi} \cdot \nabla u dx \right| \notag\\
&\leq \frac{1}{2} \| \nabla u \|_{L_w^2(\Omega)}^2 + C \| u \|_{L_w^2(\Omega)}^2 \left( \|\bar{u}\|_{L_w^2(\Omega)}^{2/\theta} + \|\bar{v}\|_{L_{1/w}^2(\Omega)}^{2/\theta} + \|\bar{u}\|_{L_w^2(\Omega)}^2 + \|\bar{v}\|_{L_{1/w}^2(\Omega)}^2 \right) \notag\\
&\leq \frac{1}{2} \| \nabla u \|_{L_w^2(\Omega)}^2 + C \| u \|_{L_w^2(\Omega)}^2 \left( \|\bar{u}\|_{L_w^2(\Omega)}^{2/\theta} + \|\bar{v}\|_{L_{1/w}^2(\Omega)}^{2/\theta} + 1 \right), \label{id_07}\\
&\quad\ \left| \int_\Omega w^{-1} v \nabla \bar{\psi} \cdot \nabla v dx \right| \notag\\
&\leq \frac{1}{2} \| \nabla v \|_{L_{1/w}^2(\Omega)}^2 + C \| v \|_{L_{1/w}^2(\Omega)}^2 \left( \|\bar{u}\|_{L_w^2(\Omega)}^{2/\theta} + \|\bar{v}\|_{L_{1/w}^2(\Omega)}^{2/\theta} + \|\bar{u}\|_{L_w^2(\Omega)}^2 + \|\bar{v}\|_{L_{1/w}^2(\Omega)}^2 \right) \notag\\
&\leq \frac{1}{2} \| \nabla v \|_{L_{1/w}^2(\Omega)}^2 + C \| v \|_{L_{1/w}^2(\Omega)}^2 \left( \|\bar{u}\|_{L_w^2(\Omega)}^{2/\theta} + \|\bar{v}\|_{L_{1/w}^2(\Omega)}^{2/\theta} + 1 \right). \label{id_08}
\end{align}
Plugging (\ref{id_07})--(\ref{id_08}) into (\ref{id_01})--(\ref{id_02}), we obtain
\begin{align}
&\quad\ \frac{d}{dt} \left( \| u \|_{L_w^2(\Omega)}^2 + \| v \|_{L_{1/w}^2(\Omega)}^2 \right) + \| \nabla u \|_{L_w^2(\Omega)}^2 + \| \nabla v \|_{L_{1/w}^2(\Omega)}^2 \notag\\
&\leq C \left( \| u \|_{L_w^2(\Omega)}^2 + \| v \|_{L_{1/w}^2(\Omega)}^2 \right) \left( \|\bar{u}\|_{L_w^2(\Omega)}^{2/\theta} + \|\bar{v}\|_{L_{1/w}^2(\Omega)}^{2/\theta} + 1 \right). \label{id_09}
\end{align}
By Gronwall's inequality,
\begin{align}
&\quad\ \sup_{t \in [0,T]}\left( \| u \|_{L_w^2(\Omega)}^2 + \| v \|_{L_{1/w}^2(\Omega)}^2 \right) \notag\\
&\leq \left( \| u_0 \|_{L_w^2(\Omega)}^2 + \| v_0 \|_{L_{1/w}^2(\Omega)}^2 \right) \exp \left\{ C\int_0^T \left( \|\bar{u}\|_{L_w^2(\Omega)}^{2/\theta} + \|\bar{v}\|_{L_{1/w}^2(\Omega)}^{2/\theta} + 1 \right) dt \right\}. \label{id_10}
\end{align}
Inequality (\ref{id_10}) implies
\begin{align*}
&\quad\ \int_0^T \left( \| u \|_{L_w^2(\Omega)}^{2/\theta} + \| v \|_{L_{1/w}^2(\Omega)}^{2/\theta} \right) dt \\
&\leq T \left( \| u_0 \|_{L_w^2(\Omega)}^2 + \| v_0 \|_{L_{1/w}^2(\Omega)}^2 \right)^{1/\theta} \exp \left\{ \frac{C}{\theta} \int_0^T \left( \|\bar{u}\|_{L_w^2(\Omega)}^{2/\theta} + \|\bar{v}\|_{L_{1/w}^2(\Omega)}^{2/\theta} + 1 \right) dt \right\},
\end{align*}
which gives
\begin{align*}
\| (u,v) \|_X &\leq T^{\theta/2} \left( \| u_0 \|_{L_w^2(\Omega)}^2 + \| v_0 \|_{L_{1/w}^2(\Omega)}^2 \right)^{1/2} \exp \left\{ \frac{C}{2} \left( \| (\bar{u},\bar{v}) \|_X^{2/\theta} + T \right) \right\} \\
&\leq T^{\theta/2} \left( \| u_0 \|_{L_w^2(\Omega)}^2 + \| v_0 \|_{L_{1/w}^2(\Omega)}^2 \right)^{1/2} \exp \left\{ \frac{C}{2} \left( 1 + T \right) \right\}
\end{align*}
for any $(\bar{u},\bar{v}) \in B_1(0) \subset X$. Therefore, we can choose $T$ sufficiently small so that $F$ maps the unit ball $B_1(0) \subset X$ into itself. Moreover, combining (\ref{id_09}) and (\ref{id_10}), it holds that $u \in L^2\left((0,T);H_w^1(\Omega)\right)$ and $v \in L^2\left((0,T);H_{1/w}^1(\Omega)\right)$, and hence $u_t \in L^2\left((0,T);H_w^{-1}(\Omega)\right)$ and $v_t \in L^2\left((0,T);H_{1/w}^{-1}(\Omega)\right)$. In view of Lemma \ref{lem_2_3}, we can apply Aubin-Lions lemma to conclude that $F(B_1(0))$ is precompact in $L^2((0,T);L_w^2(\Omega)) \times L^2((0,T);L_{1/w}^2(\Omega))$. Using this fact together with (\ref{id_10}), we conclude the compactness in $X = L^{2/\theta}\left((0,T);L_w^2(\Omega)\right) \times L^{2/\theta}\left((0,T);L_{1/w}^2(\Omega)\right)$. Therefore, by applying Schauder's fixed-point theorem, we obtain a solution of (\ref{eqn_u})--(\ref{initial_psi}).
 
Next, we are going to show the solution we obtained is unique since Schauder's fixed-point theorem does not guarantee uniqueness. Suppose that $(u_i,v_i,\psi_i)$, $i = 1,2$, are both solutions of (\ref{eqn_u})--(\ref{initial_psi}). Multiplying the difference of equations of $u_1$ and $u_2$ by $u_1 - u_2$ and using estimates analogous to (\ref{id_07})--(\ref{id_08}), we have
\begin{align}
&\quad\ \frac{1}{2} \frac{d}{dt} \| u_1 - u_2 \|_{L_w^2(\Omega)}^2 + \| \nabla (u_1 - u_2) \|_{L_w^2(\Omega)}^2 \notag\\
&= \int_\Omega w (u_1 \nabla \psi_1 - u_2 \nabla \psi_2) \cdot \nabla (u_1 - u_2) dx \notag\\
&\leq \| \nabla (u_1 - u_2) \|_{L_w^2(\Omega)} \| u_1 - u_2 \|_{L_w^{2\sqrt{2} / (\sqrt{2} - 1)}(\Omega)} \| \nabla \psi_1 \|_{L_w^{2\sqrt{2}}(\Omega)} \notag\\
&\quad\quad + \| \nabla (u_1 - u_2) \|_{L_w^2(\Omega)} \| u_2 \|_{L_w^{2\sqrt{2} / (\sqrt{2} - 1)}(\Omega)} \| \nabla (\psi_1 - \psi_2) \|_{L_w^{2\sqrt{2}}(\Omega)} \notag\\
&\leq \frac{1}{2} \| \nabla (u_1 - u_2) \|_{L_w^2(\Omega)}^2 + C \| u_1 - u_2 \|_{L_w^2(\Omega)}^2 \left( \| u_1 \|_{L_w^2(\Omega)}^{2/\theta} + \| v_1 \|_{L_{1/w}^2(\Omega)}^{2/\theta} + 1 \right) \label{id_11}\\
&\quad\quad + C \| u_2 \|_{H_w^1(\Omega)}^2 \left( \| u_1 - u_2 \|_{L_w^2(\Omega)}^2 + \| v_1 - v_2 \|_{L_{1/w}^2(\Omega)}^2 \right). \notag
\end{align}
Similarly, for $v_1 - v_2$, we have
\begin{align}
&\quad\ \frac{1}{2} \frac{d}{dt} \| v_1 - v_2 \|_{L_{1/w}^2(\Omega)}^2 + \| \nabla (v_1 - v_2) \|_{L_{1/w}^2(\Omega)}^2 \notag\\
&\leq \frac{1}{2} \| \nabla (v_1 - v_2) \|_{L_{1/w}^2(\Omega)}^2 + C \| v_1 - v_2 \|_{L_{1/w}^2(\Omega)}^2 \left( \| u_1 \|_{L_w^2(\Omega)}^{2/\theta} + \| v_1 \|_{L_{1/w}^2(\Omega)}^{2/\theta} + 1 \right) \label{id_12}\\
&\quad\quad + C \| v_2 \|_{H_{1/w}^1(\Omega)}^2 \left( \| u_1 - u_2 \|_{L_w^2(\Omega)}^2 + \| v_1 - v_2 \|_{L_{1/w}^2(\Omega)}^2 \right). \notag
\end{align}
Combining (\ref{id_11}) and (\ref{id_12}), we obtain
\begin{align}
&\quad\ \frac{d}{dt} \left( \| u_1 - u_2 \|_{L_w^2(\Omega)}^2  + \| v_1 - v_2 \|_{L_{1/w}^2(\Omega)}^2 \right) + \| \nabla(u_1 - u_2) \|_{L_w^2(\Omega)}^2  + \| \nabla(v_1 - v_2) \|_{L_{1/w}^2(\Omega)}^2 \notag\\
&\leq C \left( \| u_1 - u_2 \|_{L_w^2(\Omega)}^2  + \| v_1 - v_2 \|_{L_{1/w}^2(\Omega)}^2 \right) \label{id_13}\\
&\quad\quad \cdot \left( \| u_1 \|_{L_w^2(\Omega)}^{2/\theta} + \| v_1 \|_{L_{1/w}^2(\Omega)}^{2/\theta} + \| u_2 \|_{H_w^1(\Omega)}^2 + \| v_2 \|_{H_{1/w}^1(\Omega)}^2 + 1 \right). \notag
\end{align}
Since $\| u_1 \|_{L_w^2(\Omega)}^{2/\theta} + \| v_1 \|_{L_{1/w}^2(\Omega)}^{2/\theta} + \| u_2 \|_{H_w^1(\Omega)}^2 + \| v_2 \|_{H_{1/w}^1(\Omega)}^2 + 1 \in L^1((0,T))$, by Gronwall's inequality, (\ref{id_13}) implies that $(u_1,v_1) \equiv (u_2,v_2)$, and it follows that $\psi_1 \equiv \psi_2$.

To complete the proof, it suffices to show that the $u,v \geq 0$ for the solution $(u,v,\psi)$. Multiplying the equations of $u$ and $v$ by $u_- := \max\{-u,0\}$ and $v_- := \max\{-v,0\}$, respectively, and summing up, we get
\begin{align}
&\quad\ \frac{1}{2} \frac{d}{dt} \left( \| u_- \|_{L_w^2(\Omega)}^2 + \| v_- \|_{L_{1/w}^2(\Omega)}^2 \right) + \| \nabla u_- \|_{L_w^2(\Omega)}^2 + \| \nabla v_- \|_{L_{1/w}^2(\Omega)}^2 \notag\\
&= \int_\Omega w u_- \nabla \psi \cdot \nabla u_- dx - \int_\Omega w^{-1} v_- \nabla \psi \cdot \nabla v_- dx. \label{id_14}
\end{align}
Again, by repeating the argument for (\ref{id_09}), (\ref{id_14}) implies
\begin{align*}
&\quad\ \frac{d}{dt} \left( \| u_- \|_{L_w^2(\Omega)}^2 + \| v_- \|_{L_{1/w}^2(\Omega)}^2 \right) + \| \nabla u_- \|_{L_w^2(\Omega)}^2 + \| \nabla v_- \|_{L_{1/w}^2(\Omega)}^2 \notag\\
&\leq C \left( \| u_- \|_{L_w^2(\Omega)}^2 + \| v_- \|_{L_{1/w}^2(\Omega)}^2 \right) \left( \|u\|_{L_w^2(\Omega)}^{2/\theta} + \|v\|_{L_{1/w}^2(\Omega)}^{2/\theta} + 1 \right).
\end{align*}
Therefore, Gronwall's inequality implies that $u_- \equiv v_- \equiv 0$, that is, $u,v \geq 0$.

\section{Global existence}\label{sec_3}

In this section, we prove Theorem \ref{global_wellposedness}.

\begin{proof}[Proof of Theorem \ref{global_wellposedness}]
Since the local well-posedness is established by Theorem \ref{local_wellposedness}, it suffices to prove uniform-in-time bounds for $\| u \|_{L_w^2(\Omega)}$ and $\| v \|_{L_{1/w}^2(\Omega)}$. Multiplying the equation (\ref{new_eqn_u}) by $u$, integrating over $\Omega$ and doing integration by parts yield
\begin{align*}
\frac{1}{2} \frac{d}{dt} \int_\Omega wu^2 dx &= -\int_\Omega w|\nabla u|^2 dx + \int_\Omega w u \nabla \psi \cdot \nabla u dx.
\end{align*}
By H\"{o}lder's inequality and the analogous arguments as in (\ref{id_03}) and (\ref{id_16}), we have
\begin{align*}
\left| \int_\Omega w u \nabla \bar{\psi} \cdot \nabla u dx \right| &\leq \| \nabla u \|_{L_w^2(\Omega)} \| u \|_{L_w^{2\sqrt{2} / (\sqrt{2} - 1)}(\Omega)} \| \nabla \psi \|_{L_w^{2\sqrt{2}}(\Omega)}, \\
&\lesssim \| \nabla u \|_{L_w^2(\Omega)} \| u \|_{H_w^1(\Omega)}^{1-\theta} \| u \|_{L_w^2(\Omega)}^\theta \left( \| u \|_{L^2_w(\Omega)} + \| v \|_{L^2_{1/w}(\Omega)} \right).
\end{align*}
Therefore,
\begin{align}\label{id_20}
&\quad\ \frac{1}{2} \frac{d}{dt} \int_\Omega wu^2 dx + \int_\Omega w|\nabla u|^2 dx \notag\\
&\lesssim \| \nabla u \|_{L_w^2(\Omega)} \| u \|_{H_w^1(\Omega)}^{1-\theta} \| u \|_{L_w^2(\Omega)}^\theta \left( \| u \|_{L^2_w(\Omega)} + \| v \|_{L^2_{1/w}(\Omega)} \right).
\end{align}
Similarly, for $v$, we have
\begin{align}\label{id_21}
&\quad\ \frac{1}{2} \frac{d}{dt} \int_\Omega w^{-1} v^2 dx + \int_\Omega w^{-1} |\nabla v|^2 dx \notag\\
&\lesssim  \| \nabla v \|_{L_{1/w}^2(\Omega)} \| v \|_{H_{1/w}^1(\Omega)}^{1-\theta} \| v \|_{L_{1/w}^2(\Omega)}^\theta \left( \| u \|_{L^2_w(\Omega)} + \| v \|_{L^2_{1/w}(\Omega)} \right).
\end{align}
Combining (\ref{id_20})--(\ref{id_21}) and using Young's inequality yield
\begin{align}
&\quad\ \frac{d}{dt} \left( \| u \|_{L_w^2(\Omega)}^2 + \| v \|_{L_{1/w}^2(\Omega)}^2 \right) + \| \nabla u \|_{L_w^2(\Omega)}^2 + \| \nabla v \|_{L_{1/w}^2(\Omega)}^2 \notag\\
&\lesssim \left( \| u \|_{L_w^2(\Omega)}^2 + \| v \|_{L_{1/w}^2(\Omega)}^2 \right) \left( \|u\|_{L_w^2(\Omega)}^{2/\theta} + \|v\|_{L_{1/w}^2(\Omega)}^{2/\theta} + \|u\|_{L_w^2(\Omega)}^2 + \|v\|_{L_{1/w}^2(\Omega)}^2 \right) \notag\\
&\lesssim \left( \| u \|_{L_w^2(\Omega)}^2 + \| v \|_{L_{1/w}^2(\Omega)}^2 \right)^{1 + \frac{1}{\theta}} + \left( \| u \|_{L_w^2(\Omega)}^2 + \| v \|_{L_{1/w}^2(\Omega)}^2 \right)^2. \label{id_22}
\end{align}
By H\"{o}lder's inequality and Lemma \ref{lem_2_1}, we have
\begin{align*}
\| u - u_{\Omega,w} \|_{L_w^2(\Omega)} \lesssim \| \nabla u \|_{L_w^2(\Omega)},
\end{align*}
where
\begin{align*}
u_{\Omega,w} = \left( \int_\Omega w(x) dx \right)^{-1} \int_\Omega u(x) w(x) dx.
\end{align*}
The last inequality implies
\begin{align}\label{id_31}
\| u \|_{L_w^2(\Omega)}^2 \lesssim \| \nabla u \|_{L_w^2(\Omega)}^2 + \| u \|_{L_w^1(\Omega)}^2.
\end{align}
Similarly,
\begin{align}\label{id_32}
\| v \|_{L_{1/w}^2(\Omega)}^2 \lesssim \| \nabla v \|_{L_{1/w}^2(\Omega)}^2 + \| v \|_{L_{1/w}^1(\Omega)}^2.
\end{align}
Using test functions $\xi \equiv 1$ and $\eta \equiv 1$ in Definition \ref{def_sol}, it gives that the total concentrations of individual ions are conserved, that is,
\begin{align*}
\| u(\cdot,t) \|_{L^1_w(\Omega)} = \| u_0 \|_{L^1_w(\Omega)} \quad \mbox{and} \quad \| v(\cdot,t) \|_{L^1_{1/w}(\Omega)} = \| v_0 \|_{L^1_{1/w}(\Omega)}
\end{align*}
for almost all $t > 0$. By H\"{o}lder's inequality,
\begin{align}
\| u(\cdot,t) \|_{L^1_w(\Omega)}^2 + \| v(\cdot,t) \|_{L^1_{1/w}(\Omega)}^2 &= \| u_0 \|_{L^1_w(\Omega)}^2 + \| v_0 \|_{L^1_{1/w}(\Omega)}^2 \notag\\
&\leq \| u_0 \|_{L^2_w(\Omega)}^2 \int_\Omega w dx + \| v_0 \|_{L^2_{1/w}(\Omega)}^2 \int_\Omega w^{-1} dx \notag\\
&\leq C\left( \| u_0 \|_{L^2_w(\Omega)}^2 + \| v_0 \|_{L^2_{1/w}(\Omega)}^2 \right) \leq C\epsilon_1, \label{id_24}
\end{align}
provided that
\begin{align*}
\| u_0 \|_{L^2_w(\Omega)}^2 + \| v_0 \|_{L^2_{1/w}(\Omega)}^2 \leq \epsilon_1.
\end{align*}
Combining (\ref{id_31})--(\ref{id_24}), we obtain
\begin{align}\label{id_33}
\| \nabla u \|_{L_w^2(\Omega)}^2 + \| \nabla v \|_{L_{1/w}^2(\Omega)}^2 \geq \frac{1}{C} \left( \| u \|_{L_w^2(\Omega)}^2 + \| v \|_{L_{1/w}^2(\Omega)}^2 \right) - C\epsilon_1.
\end{align}
Let
\begin{align*}
H(t) = \| u \|_{L_w^2(\Omega)}^2 + \| v \|_{L_{1/w}^2(\Omega)}^2.
\end{align*}
Putting (\ref{id_33}) into (\ref{id_22}), we conclude that
\begin{align}\label{id_25}
\frac{d}{dt} H \leq - \frac{1}{C} H + C \left( H^{1 + \frac{1}{\theta}} + H^2 + \epsilon_1 \right) =: \Gamma(H).
\end{align}
By choosing $\epsilon_1$ sufficiently small, we have that the polynomial $\Gamma$ admits negative values on an interval, and then we define
\begin{align*}
\epsilon_2 := \sup \{ s > 0 : \Gamma(s) < 0 \} > 0.
\end{align*}
Therefore, we can let
\begin{align*}
\epsilon := \min \{ \epsilon_1 , \epsilon_2 \}.
\end{align*}
The inequality (\ref{id_25}) and the assumption $H(0) < \epsilon$ then imply that $H(t) < \epsilon$ for all $t \geq 0$. The proof is completed.
\end{proof}

\section{Appendix}

In the appendix, we present a sketch of the proof of the existence and uniqueness for system (\ref{eqn_u_bar})--(\ref{BC_v_bar}) with initial data (\ref{initial_u})--(\ref{initial_v}).

The proof of the existence is based on Galerkin's method. We choose $\{ \varphi_k \}_{k=1}^\infty$, which is an orthogonal basis of $H_w^1(\Omega)$ and an orthonormal basis of $L_w^2(\Omega)$, and $\{ \widetilde{\varphi}_k \}_{k=1}^\infty$, which is an orthogonal basis of $H_{1/w}^1(\Omega)$ and an orthonormal basis of $L_{1/w}^2(\Omega)$. Then we can consider the following approximate solutions
\begin{align*}
u_N(x,t) &= \sum_{k=1}^N \beta_{N,k}(t) \varphi_k (x), \\
v_N(x,t) &= \sum_{k=1}^N \widetilde{\beta}_{N,k}(t) \widetilde{\varphi}_k (x),
\end{align*}
where $\beta_{N,k}$ and $\widetilde{\beta}_{N,k}$ are determined by
\begin{align}
\dfrac{d}{dt} \beta_{N,k} &= \dfrac{d}{dt} \left\langle u_N , \varphi_k \right\rangle_w = - \left\langle \nabla u_N - u_N \nabla \bar{\psi} , \nabla \varphi_k \right\rangle_w \notag\\
&= -\| \nabla \varphi_k \|_{L_w^2(\Omega)} \beta_{N,k} + \sum_{l=1}^N \left\langle \varphi_l \nabla \bar{\psi} , \nabla \varphi_k \right\rangle_w \beta_{N,l}, \label{a_01}\\
\dfrac{d}{dt} \widetilde{\beta}_{N,k} &= \dfrac{d}{dt} \left\langle v_N , \widetilde{\varphi}_k \right\rangle_{1/w} = - \left\langle \nabla v_N + v_N \nabla \bar{\psi} , \nabla \widetilde{\varphi}_k \right\rangle_{1/w} \notag\\
&= -\| \nabla \widetilde{\varphi}_k \|_{L_{1/w}^2(\Omega)} \widetilde{\beta}_{N,k} - \sum_{l=1}^N \left\langle \widetilde{\varphi}_l \nabla \bar{\psi} , \nabla \widetilde{\varphi}_k \right\rangle_{1/w} \widetilde{\beta}_{N,l}, \label{a_02}
\end{align}
for $t \in (0,T)$, and 
\begin{align}
\beta_{N,k} (0) &= \left\langle u_0 , \varphi_k \right\rangle_w, \label{a_03}\\
\widetilde{\beta}_{N,k} (0) &= \left\langle v_0 , \widetilde{\varphi}_k \right\rangle_{1/w}, \label{a_04}
\end{align}
for all $k = 1, ..., N$. Here $\langle \cdot , \cdot \rangle_w$ and $\langle \cdot , \cdot \rangle_{1/w}$ denote inner products in $L_w^2(\Omega)$ and $L_{1/w}^2(\Omega)$, respectively. By similar estimates as in (\ref{id_18}) and (\ref{id_19}), coefficients in the equations (\ref{a_01}) and (\ref{a_02}) are integrable on $[0,T]$, which implies the unique solvability of (\ref{a_01})--(\ref{a_04}). Moreover, multiplying (\ref{a_01}) and (\ref{a_02}) by $\beta_{N,k}$ and $\widetilde{\beta}_{N,k}$ respectively, for each $k = 1, ..., N$, and summing the resulting identities, we obtain
\begin{align*}
\frac{1}{2} \frac{d}{dt} \| u_N \|_{L_w^2(\Omega)} &= - \left\langle \nabla u_N - u_N \nabla \bar{\psi} , \nabla u_N \right\rangle_w, \\
\frac{1}{2} \frac{d}{dt} \| v_N \|_{L_{1/w}^2(\Omega)} &= - \left\langle \nabla v_N + v_N \nabla \bar{\psi} , \nabla v_N \right\rangle_{1/w}.
\end{align*}
Using the same argument for (\ref{id_09}), we have
\begin{align*}
&\quad\ \frac{d}{dt} \left( \| u_N \|_{L_w^2(\Omega)}^2 + \| v_N \|_{L_{1/w}^2(\Omega)}^2 \right) + \| \nabla u_N \|_{L_w^2(\Omega)}^2 + \| \nabla v_N \|_{L_{1/w}^2(\Omega)}^2 \notag\\
&\leq C \left( \| u_N \|_{L_w^2(\Omega)}^2 + \| v_N \|_{L_{1/w}^2(\Omega)}^2 \right) \left( \|\bar{u}\|_{L_w^2(\Omega)}^{2/\theta} + \|\bar{v}\|_{L_{1/w}^2(\Omega)}^{2/\theta} + 1 \right),
\end{align*}
By Gronwall's inequality, the last estimate gives $L^\infty L_w^2 \cap L^2 H_w^1$ and $L^\infty L_{1/w}^2 \cap L^2 H_{1/w}^1$ uniform bounds for $u_N$ and $v_N$, respectively. Therefore, there exists a subsequence, still denoted by $( u_N , v_N)$, such that $u_N \rightharpoonup u$ on $L^2\left((0,T);H_w^1(\Omega)\right)$ and $v_N \rightharpoonup v$ on $L^2\left((0,T);H_{1/w}^1(\Omega)\right)$ weakly as $N \rightarrow \infty$. It is straightforward to see that $(u,v)$ is a solution to the system (\ref{eqn_u_bar})--(\ref{BC_v_bar}) with initial data (\ref{initial_u})--(\ref{initial_v}). As for the uniqueness, it can be proved by following the same lines as in (\ref{id_11})--(\ref{id_13}).


\begin{thebibliography}{10}

\bibitem{BD00} {\sc P. Biler and J. Dolbeault}, {\em Long time behavior of solutions of Nernst-Planck and Debye-H\"{u}ckel drift-diffusion systems}, Ann. Henri Poincar\'{e}, 1 (2000), 461--472.

\bibitem{BHN94} {\sc P. Biler, W. Hebisch and T. Nadzieja}, {\em The Debye system: existence and large time behavior of solutions}, Nonlinear Anal. 23 (1994), no. 9, 1189--1209.

\bibitem{BN97} {\sc P, Biler and T. Nadzieja}, {\em A singular problem in electrolytes theory}, Math. Methods Appl. Sci. 20 (1997), no. 9, 767--782.

\bibitem{CSH17} {\sc J. Cartailler, Z. Schuss and D. Holcman}, {\em Analysis of the Poisson-Nernst-Planck equation in a ball for modeling the Voltage-Current relation in neurobiological microdomains}, Phys. D 339 (2017), 39--48.

\bibitem{CLE97} {\sc D. Chen, J. Lear and B. Eisenberg}, {\em Permeation through an open channel: Poisson-Nernst-Planck theory of a synthetic ionic channel}, Biophys J. (1997), 72(1), 97--116.

\bibitem{C92} {\sc S.-K. Chua}, {\em Extension theorems on weighted Sobolev spaces}, Indiana Univ. Math. J. 41 (1992), no. 4, 1027--1076.

\bibitem{CI19} {\sc P. Constantin and M. Ignatova}, {\em On the Nernst-Planck-Navier-Stokes system}, Arch. Ration. Mech. Anal. 232 (2019), no. 3, 1379--1428. 

\bibitem{DST08} {\sc R. G. Dur\'{a}n, M. Sanmartino and M. Toschi}, {\em Weighted a priori estimates for the Poisson equation}, Indiana Univ. Math. J. 57 (2008), no. 7, 3463--3478.

\bibitem{E98} {\sc B. Eisenberg}, {\em Ionic channels in biological membranes: natural nanotubes}, Acc. Chem. Res., 31 (1998), 117--123.

\bibitem{FKS82} {\sc E. B. Fabes, C. E. Kenig and R. P. Serapioni}, {\em The local regularity of solutions of degenerate elliptic equations}, Comm. Partial Differential Equations 7 (1982), no. 1, 77--116.

\bibitem{F00} {\sc A. Fr\"{o}hlich}, {\em The Helmholtz decomposition of weighted $L^q$-spaces for Muckenhoupt weights}, Ann. Univ. Ferrara Sez. VII (N.S.) 46 (2000), 11--19.

\bibitem{G85}{\sc H. Gajewski}, {\em On existence, uniqueness and asymptotic behavior of solutions of the basic equations for carrier transport in semiconductors}, Z. Angew. Math. Mech. 65 (1985), no. 2, 101--108.

\bibitem{GG86} {\sc H. Gajewski and K. Gr\"{o}ger}, {\em On the basic equations for carrier transport in semiconductors}, J. Math. Anal. Appl. 113 (1986), no. 1, 12--35.

\bibitem{GNE02} {\sc D. Gillespie, W. Nonner and R. S. Eisenberg}, {\em Coupling Poisson-Nernst-Planck and density functional theory to calculate ion flux}, J. Phys.: Condens. Matter, 14 (2002), 12129--12145.

\bibitem{HKM06} {\sc J. Heinonen, T. Kilpel\"{a}inen and O. Martio}, {\em Nonlinear potential theory of degenerate elliptic equations}, Unabridged republication of the 1993 original, Dover Publications, Inc., Mineola, NY, 2006.

\bibitem{H01} {\sc B. Hille}, {\em Ion Channels of Excitable Membranes}, 3rd Edition, Sinauer Associates, Inc., 2001.

\bibitem{H19} {\sc C.-Y. Hsieh}, {\em Stability of radial symmetric solutions of the Poisson-Nernst-Planck system in annular domains}, Discrete Contin. Dyn. Syst. Ser. B 24 (2019), no. 6, 2657--2681.

\bibitem{HHLLL15} {\sc C.-Y. Hsieh, Y. Hyon, H. Lee, T.-C. Lin and C. Liu}, {\em Transport of charged particles: Entropy production and maximum dissipation principle}, J. Math. Anal. Appl., 422 (2015), 309--336.

\bibitem{HL15} {\sc C.-Y. Hsieh and T.-C. Lin}, {\em Exponential decay estimates for the stability of boundary layer solutions to Poisson-Nernst-Planck systems: one spatial dimension case}, SIAM J. Math. Anal. 47 (2015), no. 5, 3442--3465.

\bibitem{HLLL20} {\sc C.-Y. Hsieh, T.-C. Lin, C. Liu and P. Liu}, {\em Global existence of the non-isothermal Poisson-Nernst-Planck-Fourier system}, J. Differential Equations 269 (2020), no. 9, 7287--7310.

\bibitem{HY20} {\sc C.-Y. Hsieh and Y. Yu}, {\em Debye layer in Poisson-Boltzmann model with isolated singularities}, Arch. Ration. Mech. Anal. 236 (2020), no. 1, 289--327.

\bibitem{KBA07} {\sc M. S. Kilic, M. Z. Bazant and A. Ajdari}, {\em Steric effects in the dynamics of electrolytes at large applied voltages. I. Double-layer charging}, Phys. Rev. E, 75 (2007), 021502.

\bibitem{KBA07-2} {\sc M. S. Kilic, M. Z. Bazant and A. Ajdari}, {\em Steric effects in the dynamics of electrolytes at large applied voltages. II. Modified Poisson-Nernst-Planck equations}, Phys. Rev. E, 75 (2007), 021503.

\bibitem{LSU68} {\sc O. A. Lady\v{z}enskaja, V. A. Solonnikov and N. N. Uralʹceva}, {\em Linear and quasilinear equations of parabolic type}, American Mathematical Society, Providence, R.I. 1968.

\bibitem{Lee16} {\sc C.-C. Lee}, {\em Asymptotic analysis of charge conserving Poisson-Boltzmann equations with variable dielectric coefficients}, Discrete Contin. Dyn. Syst. 36 (2016), no. 6, 3251--3276.

\bibitem{LLHLL11} {\sc C.-C. Lee, H. Lee, Y. Hyon, T.-C. Lin and C. Liu}, {\em New Poisson-Boltzmann type equations: one-dimensional solutions}, Nonlinearity 24 (2011), no. 2, 431--458.

\bibitem{LR18} {\sc C.-C. Lee and R. J. Ryham}, {\em Boundary asymptotics for a non-neutral electrochemistry model with small Debye length}, Z. Angew. Math. Phys. 69 (2018), no. 2, Art. 41, 13 pp.

\bibitem{LE14} {\sc T.-C. Lin and B. Eisenberg}, {\em A new approach to the Lennard-Jones potential and a new model: PNP-steric equations}, Commun. Math. Sci., 12 (2014), 149--173.

\bibitem{LJX18} {\sc P. Liu, X. Ji and Z. Xu}, {\em Modified Poisson-Nernst-Planck model with accurate Coulomb correlation in variable media}, SIAM J. Appl. Math., 78 (2018), 226--245.

\bibitem{MRS90} {\sc P. A. Markowich, C. A. Ringhofer and C. Schmeiser}, {\em Semiconductor Equations}, Springer-Verlag, Vienna, 1990.

\bibitem{M75} {\sc M. S. Mock}, {\em Asymptotic behavior of solutions of transport equations for semiconductor devices}, J. Math. Anal. Appl. 49 (1975), 215--225.

\bibitem{MJP07} {\sc Y. Mori, J. W. Jerome and C. S. Peskin}, {\em A three-dimensional model of cellular electrical activity}, Bull. Inst. Math. Acad. Sin. (N.S.) 2 (2007), no. 2, 367--390.

\bibitem{M72} {\sc B. Muckenhoupt}, {\em Weighted norm inequalities for the Hardy maximal function}, Trans. Amer. Math. Soc. 165 (1972), 207--226.

\bibitem{M74} {\sc B. Muckenhoupt}, {\em The equivalence of two conditions for weight functions}, Studia Math. 49 (1974), 101--106.

\bibitem{R19} {\sc M. Rathmair}, {\em On how Poincar\'{e} inequalities imply weighted ones}, Monatsh. Math. 188 (2019), no. 4, 753--763. 

\bibitem{RCA89} {\sc O. J. Riveros, T. L. Croxton and W. M. Armstrong},
{\em Liquid junction potentials calculated from numerical solutions of the Nernst-Planck and Poisson equations}, J. Theor. Biol., 140 (1989), 221--230.

\bibitem{RLW06} {\sc R. Ryham, C. Liu and Z.-Q. Wang}, {\em On electro-kinetic fluids: one dimensional configurations}, Discrete Contin. Dyn. Syst. Ser. B 6 (2006), no. 2, 357--371.

\bibitem{RLZ07} {\sc R. Ryham, C. Liu and L. Zikatanov}, {\em Mathematical models for the deformation of electrolyte droplets}, Discrete Contin. Dyn. Syst. Ser. B, 8 (2007), no. 3, 649--661.

\bibitem{S93} {\sc E. M. Stein}, {\em Harmonic analysis: real-variable methods, orthogonality, and oscillatory integrals}, With the assistance of Timothy S. Murphy. Princeton Mathematical Series, 43. Monographs in Harmonic Analysis, III. Princeton University Press, Princeton, NJ, 1993.

\bibitem{WXLLS14} {\sc L. Wan, S. Xu, M. Liao, C. Liu and P. Sheng}, {\em Self-consistent approach to global charge neutrality in electrokinetics: a surface potential trap model}, Phys. Rev. X, 4 (2014), 011042.

\bibitem{WLT16} {\sc Y. Wang, C. Liu and Z. Tan}, {\em A generalized Poisson-Nernst-Planck-Navier-Stokes model on the fluid with the crowded charged particles: derivation and its well-posedness}, SIAM J. Math. Anal., 48 (2016), 3191--3235.

\bibitem{XSL14} {\sc S. Xu, P. Sheng and C. Liu}, {\em An energetic variational approach for ion transport}, Commun. Math. Sci., 12 (2014), 779--789.

\bibitem{ZGLWS08} {\sc J. Zhang, X. Gong, C. Liu, W. Wen and P. Sheng}, {\em Electrorheological fluid dynamics}, Phys. Rev. Lett., 101 (2008), 194503.

\end{thebibliography}
\end{document}